\documentclass[12pt,letterpaper]{amsart}
\usepackage{geometry}
\usepackage{amsmath,amssymb,amsfonts,amsthm}
\usepackage{latexsym,url,setspace,hyperref,mathpazo}
\newtheorem{theorem}{Theorem}
\newtheorem{lemma}{Lemma}
\newtheorem{corollary}{Corollary}
\theoremstyle{remark}
\newtheorem{remark}{Remark}

\newcommand{\C}{\mathbb{C}}
\newcommand{\disk}{\mathbb{D}}
\newcommand{\D}{\Omega}

\newcommand{\im}{\text{Im}}
\newcommand{\Dc}{\overline{\Omega}}
\newcommand{\dbar}{\overline{\partial}}
\newcommand{\ds}{\displaystyle}
\onehalfspace
\title[Compactness of products of Hankel operators]{Compactness of products of
Hankel operators on convex Reinhardt domains in $\C^2$}
\author{\u{Z}eljko \u{C}u\u{c}kovi\'c}
\author{S\"{o}nmez \c{S}ahuto\u{g}lu}
\email{Zeljko.Cuckovic@utoledo.edu, Sonmez.Sahutoglu@utoledo.edu}
\address{University of Toledo, Department of Mathematics \& Statistics, 
Toledo, OH 43606, USA}
\date{\today}
\begin{document}

\begin{abstract}
Let $\D$ be a piecewise smooth bounded convex Reinhardt  domain in $\C^2.$ Assume
that  the  symbols $\phi$ and $\psi$ are  continuous on $\Dc$  and harmonic on the disks
in the boundary of $\D.$ We show that if the product of Hankel operators  $H^*_{\psi}
H_{\phi}$  is compact on the Bergman space of $\D,$  then on any disk in the boundary of
$\D,$ either $\phi$ or $\psi$ is holomorphic.
\end{abstract}

\maketitle

This paper is a sequel to our two previous  papers
\cite{CuckovicSahutoglu09,CuckovicSahutoglu10} on compactness of Hankel
operators on Bergman spaces of domains in $\C^n.$ In the first paper we
studied compactness of a single Hankel operator with a smooth symbol on quite
general domains. We note that in this paper smooth means $C^{\infty}$-smooth. We used
$\dbar$ methods to relate the compactness property of Hankel operators to the behavior of
the symbol on the analytic disks in the boundary of the domain. The most complete result
is the following theorem in $\C^2.$ Here $H_{\phi}$ denotes the Hankel operator on the
Bergman space $A^2(\D)$ with a symbol $\phi.$ Furthermore, $\partial \D$  and $\mathbb{D}$
denote the boundary of $\D$ and  the open unit disk in the complex plane,
respectively. 

\begin{theorem}[\cite{CuckovicSahutoglu09}]\label{Thm1}
Let $\D$ be a smooth bounded convex domain in $\C^2$ and $\phi \in C^{\infty}(\Dc).$ 
Then $H_{\phi}$ is compact on  $A^2 (\D)$ if and only if $\phi\circ f$ is holomorphic for
any holomorphic mapping $f:\mathbb{D}\to \partial \D.$
\end{theorem}

In the second paper we studied compactness of products of two Hankel
operators on the polydisk. Notable is the absence of  $\dbar$ methods: the
domain is simple enough to be treated by reducing the dimension by one.  For
simplicity, we state the main result in $\C^2$ only. 

\begin{theorem}[\cite{CuckovicSahutoglu10}]\label{Thm2}
Let $\D$ be the bidisk in $\C^2$ and the symbols $\phi,\psi\in C (\Dc)$ such
that $\phi\circ f$ and $\psi\circ f$ are harmonic for any holomorphic mapping 
$f:\mathbb{D}\to \partial \D.$ Then $H^*_{\psi} H_{\phi}$ is compact on $A^2
(\D)$ if and only if for any holomorphic function $f:\mathbb{D}\to \partial \D,$
either $\phi\circ f$ or $\psi\circ f$ is holomorphic.  
\end{theorem}

In this paper we treat  domains that are  more general  than a polydisk (see Theorem
\ref{ThmMain}). A domain $\D\subset \C^n$ is called Reinhardt if $(z_1,\ldots,z_n)\in \D$
and $\theta_1,\ldots,\theta_n\in \mathbb{R}$ imply that
$(e^{i\theta_1}z_1,\ldots,e^{i\theta_n}z_n)\in \D.$ Namely, the domain $\D$ is
circular in each variable.  The ball and the polydisk are the best 
known examples of Reinhardt domains. 

The following theorem is the main result of our paper. As before, the analyticity of the
symbols is a necessary condition for compactness of the product of Hankel operators,
provided that their symbols are harmonic on the disks in the boundary.

\begin{theorem}\label{ThmMain}
Let $\D$ be a piecewise smooth bounded convex Reinhardt  domain in $\C^2.$ Assume
that the  symbols $\phi,\psi\in C(\Dc)$ are such that $\phi\circ f$ and $\psi\circ f$ are 
harmonic for every holomorphic function $f:\disk \to \partial\D.$ If $H^*_{\psi} H_{\phi}$
is compact on $A^2 (\D)$ then for  every holomorphic function  $f:\disk \to \partial\D$
either $\phi\circ f $ or $\psi\circ f$ is holomorphic.
\end{theorem}

The proof of Theorem \ref{ThmMain} uses convexity and rotational symmetry of the domain in
a significant way. If there is a disk $\Delta$ in the boundary of a convex Reinhardt
domain $\D$ then there are disks in $\Dc$ nearby $\Delta$ of at least the same size.
Furthermore, these disks ``converge" to $\Delta.$ This geometric property is an important
ingredient in our proof. 

\begin{remark}
Even though Theorem \ref{Thm1} is stated for symbols that are  smooth up to
the boundary and domains with smooth boundaries, the proof shows that the theorem is still
true under reasonably weaker smoothness assumptions. In the case of the polydisk Le
\cite{Le10} studied compactness of Hankel operators with symbols continuous on the
closure of the polydisk.
\end{remark}

\begin{remark}
Products of Hankel operators can be viewed as semicommutators of Toeplitz
operators. Several authors have studied compactness of these semicommutators on
the unit disk $\disk$ and the polydisk $\disk^n.$ Zheng \cite{Zheng89}
characterized compact semicommutators of Toeplitz operators with symbols that
are harmonic on $\disk.$ Later Ding and Tang \cite{DingTang01}, Choe, Koo, and
Lee \cite{ChoeKooLee04}, and Choe, Lee, Nam, and Zheng \cite{ChoeLeeNamZheng07}
extended this result to  semicommutators of Toeplitz operator acting on the
Bergman  space of $\disk^n$ with the assumption that the symbols are
pluriharmonic functions on $\disk^n.$ Notice that the symbols in Theorem
\ref{Thm2} are assumed to be continuous up to the boundary but pluriharmonic on
the disks in the boundary of $\D$ only.
\end{remark}

\begin{remark}
The class of domains to which Theorem \ref{ThmMain} applies includes many
more domains other than the bidisk. For example, it includes the intersection
of Reinhardt domains such as $(\disk\times\disk) \cap B(0,(1+\sqrt{2})/2)$
where $B(p,r)$ denotes the ball centered at $p$ with radius $r.$
\end{remark}

\begin{remark} 
If there is no disk in the boundary of a convex domain then the $\dbar$-Neumann operator
is compact (see \cite[Theorem 1.1]{FuStraube98} or \cite[Theorem 4.26]{StraubeBook}); in
turn, this implies that the Hankel operator with a symbol that is continuous on
the closure of the domain is compact (see \cite[Proposition 4.1]{StraubeBook}).
Hence, if a bounded convex domain does not have a disk in the boundary then the
product of Hankel operators with symbols continuous on the closure of the domain
is compact. For more information about Reinhardt domains we refer the reader to
\cite{JarnickiPflugBook2,KrantzBook,RangeBook}.
\end{remark}

\section*{Some Background Information and Lemmas}
Let $\D$ be a bounded domain in $\C^n$ and $A^2(\D)$ denote the Bergman space, 
the set of holomorphic functions that are square integrable on $\D$ with
respect to the Lebesgue measure $V$.  Unless we integrate on a subdomain of
$\D,$ the norm $\|. \|_{L^2(\D)}$ is denoted by $\|. \|$ and the complex inner
product $\langle .,.\rangle_{L^2(\D)}$ by $\langle .,.\rangle.$

Let $P^{\D}$ denote the Bergman projection on $\D$, the orthogonal projection
from $L^2(\D)$ onto $A^2(\D).$ The Toeplitz and Hankel operators  with symbol
$\phi\in L^{\infty}(\D)$ are defined on $A^2(\D)$ by $T^{\D}_{\phi}f=P^{\D}(\phi
f)$ and $H^{\D}_{\phi}f=\phi f-P^{\D}(\phi f),$ respectively. Notice that the
range of $H^{\D}_{\phi}$ is a subspace of the orthogonal complement of $A^2(\D)$
in $L^2(\D).$ Then one
can define the product of two Hankel operators with symbols $\psi$ and $\phi$
as $(H^{\D}_{\psi})^{*}H^{\D}_{\phi}:A^2(\D)\to A^2(\D),$ where 
$(H^{\D}_{\psi})^{*}$ denotes the Hilbert space adjoint of $H^{\D}_{\psi}.$ 
When it is clear from the context on which domain we are working on,  we will
omit the domain superscripts on the operators $P,T_{\phi},$ and $H_{\phi}.$ 

It is well known that this product can be written as a semicommutator of
Toeplitz operators. Namely, 
\begin{align}\label{Eqn1}
H^{*}_{\psi}H_{\phi}=T_{\overline{\psi}\phi}-T_{\overline{\psi}}T_{\phi}.
\end{align}
For more information about these operators we suggest the reader consult
\cite{ZhuBook,Axler88}.

We now present and prove several key lemmas that will be used in the proof of the main
theorem. They represent our idea that geometry, analysis, and approximation intertwine in
an interesting manner and they enable us to prove the main result in this paper. 

The first lemma is simple and it allows us to rewrite the product of two Hankel operators
in a different way than the semicommutator of Toeplitz operators.

\begin{lemma} \label{Lem1}
Let $\D$ be a domain in $\C^n$ and $\phi,\psi \in L^{\infty}(\D).$ Then 
$H^{*}_{\psi}H_{\phi}=PM_{\overline \psi} H_{\phi}$ where $M_{\overline \psi}$ denotes
the product by $\overline \psi.$
\end{lemma}
\begin{proof} 
Let $f,g\in A^2(\D).$ Then we have
\[\left\langle H^{*}_{\psi}H_{\phi} f,g \right\rangle = 
\left\langle H_{\phi} f,H_{\psi}g \right\rangle  =  
\left\langle H_{\phi} f,\psi g \right\rangle= 
\left\langle \overline{\psi} H_{\phi} f, g \right\rangle= 
\left\langle P \overline{\psi} H_{\phi} f, g \right\rangle.\]
Therefore,  $H^{*}_{\psi}H_{\phi}=PM_{\overline \psi} H_{\phi}.$
\end{proof}

The next lemma gives us  an important information about the disks in the boundary of
complete Reinhard domains in $\C^2.$ It shows that piecewise smooth bounded complete
Reinhardt domains in $\C^2$ can have vertical or horizontal disks only. This will allow
us to use the slicing method to approach the disks by horizontal and vertical slices of
the domain itself.  
 
\begin{lemma}\label{Lem2}
Let $\D$ be a piecewise smooth bounded complete Reinhardt domain in $\C^2$ and
let  $F=(f, g):\disk\to \partial \D$ be a holomorphic function. Then either $f$ or
$g$ is constant.
\end{lemma}
\begin{proof}
Let $F(z)=(f(z), g(z))$ be an analytic  disk in the boundary. If $|f(z)|$ and
$|g(z)|$ are constant then $F$ is constant. Therefore, there are no nontrivial
disks on the singular part of the boundary. 

Now assume that there is an analytic disk in the boundary away 
from singular points. Then we can assume that the domain is smooth and it is
given by $\rho(|z|,|w|).$ By convexity if there is a disk then it must be an
affine disk (see, for example,  \cite[Lemma 2]{CuckovicSahutoglu09} and
\cite[Proposition 3.2]{FuStraube98}). So there exist $a,b,c,d\in \C$ such that 
the set $\{(a\xi+b,c\xi+d)\in \C^2:\xi\in \disk \}$  is a disk in the boundary. We may
also assume that the disk does not intersect the coordinate axes. In other words, we may
assume that $|a\xi+b|> 0$ and $|c\xi+d|>0.$ Computing the Laplacian of
$r(\xi)=\rho(|a\xi+b|,|c\xi+d|)$ where $\xi\in \disk$ and we get 
\[0=4 \frac{\partial^2 r}{\partial \xi \partial \overline{\xi}}(\xi)
=H_{\rho}(r(\xi);W)+ \rho_x(r(\xi)) \frac{|a|^2}{|a\xi+b|}+\rho_y(r(\xi))
\frac{|c|^2}{|c\xi+d|}\]
where 
\[W=\left(a \left(\frac{\overline{a\xi+b}}{a\xi+b}\right)^{1/2},
c\left(\frac{\overline{c\xi+d}}{c\xi+d }\right)^{1/2}\right)\] 
and $H_{\rho}(p;X)$ is the (real) Hessian of $\rho$ applied to the vector $X$ at
 the point $p.$  Let $(|p|,|q|)$ be a boundary point of 
$Z=\{(x,y)\in \mathbb{R}^2:x\geq 0,y\geq 0, \rho(x,y)<0\}.$ Then the rectangle
$R_{(|p|,|q|)}\subset \mathbb{R}^2$ formed by $(0,0), (|p|,0), (0,|q|),$ and
$(|p|,|q|)$ is inside $Z$ and $(\rho_x(|p|,|q|),\rho_y(|p|,|q|))$ is normal to
the boundary of $Z$ at $(|p|,|q|).$ If $\rho_x(|p|,|q|)<0 $ and
$\rho_y(|p|,|q|)>0$ (or $\rho_x(|p|,|q|)>0 $ and $\rho_y(|p|,|q|)<0$) then the
tangential vector  to $\partial Z$ at $(|p|,|q|)$ has components with the same
sign. Then $R_{(|p|,|q|)}\cap \partial Z$ is nonempty which in turn implies
that  $R_{(|p|,|q|)}\setminus Z$ is nonempty. Similarly, if $\rho_x(|p|,|q|)<0 $
and $\rho_y(|p|,|q|)<0$ then $R_{(|p|,|q|)}$ cannot be contained in $Z.$  Hence, 
$\rho_x\geq 0,\rho_y \geq 0$ and $\rho_y+\rho_y>0,$ and 
$H_{\rho}(r(\xi);W)\geq 0$ for any $W\in \C^2.$ Therefore, either $a=0$ or
$c=0.$ That is, the disk is either horizontal or vertical.  

Assume that   $F(z)=(f(z), g(z))$ is a non-trivial analytic  disk through a
singular point in the boundary. That is, $F$ is nonconstant  and there exists 
$p\in \disk$  such that $F'(p)=0.$ Then by the previous part the smooth part of
the disk is either horizontal or vertical.  If it is horizontal then there
exists an open set $U\subset \disk$ such that $|g|$ is constant on $U.$ The
identity principle implies that $g$ is constant on $\disk.$ Hence, the whole disk
is horizontal. 
\end{proof}

As mentioned earlier, we use slicing of  the domain and the resulting disks to
approach  horizontal or vertical disks in the boundary. The following lemma
will enable us to do that in the sense  that projections of these disks onto
the complex plane approach the projection of the disk in the boundary. 
Even though this lemma is stated for horizontal disks, the result holds
for vertical disks as well.

\begin{lemma}\label{Lem3} 
Let $\D$ be a bounded convex Reinhardt domain in $\C^2$ 
and $\Delta_w= \{z\in \C:(z,w)\in \Dc\}$ for $w\in \C.$ Assume that  
$\emptyset \neq \Delta_{w_0}\times\{w_0\} \subset \partial \D$ for some $w_0\in
\C, \{w_j\}$ is a sequence of complex numbers that converges to
$w_0,$ and $\Delta_{w_j}$ is nonempty for all $j.$ Then
$\lim_{j\to \infty}r_j=r_0$ where $r_j$ denotes the radius of the
disk $\Delta_{w_j}$ for $j=0,1,2,\ldots.$
\end{lemma}
\begin{proof}
Since $\D$ is a convex Reinhardt domain it is also complete. Hence,  all of these disks
are centered at the origin and  we want to prove that $\{r_j\}$ converges to $r_0,$ the
radius of  $\Delta_{w_0}.$ In addition, since the domain is also convex one can show
that  $r_j\geq r_0$ for $j\geq 1.$ Hence $\liminf_{j\to\infty}r_j \geq r_0$

On the other hand, if $\limsup_{j\to \infty}r_j>r_0$ we can choose $p_{k} \in
\Delta_{w_{j_k}}$ such that $|p_k|=r_{j_k}$ and $\lim_{k\to
\infty}|p_{k}|=\limsup_{j\to \infty}r_j.$ Then the sequence
$\{(p_k,w_{j_k})\}\subset \partial \D$ has a subsequence that converges to a
point $(p,w_0)\in \partial \D.$ This means that  $p\in \Delta_{w_0}$ and  
\[\limsup_{j\to\infty}r_j =\lim_{k\to \infty}|p_k|=|p|\leq r_0.\] 
Therefore, $\lim_{j\to\infty}r_j = r_0.$
\end{proof}

The convergence of the disk in Lemma \ref{Lem3} brings the natural question of
a convergence of the corresponding Bergman kernels and projections.  
Let $K$ be a set in $\C^n$ and  $T_K$ denote the characteristic function of 
$K.$ That is, $T_K(z)=1$ if $z\in K$ and $T_K(z)=0$ otherwise.
Also for a function $f$ defined on a set $U$  we let $E_Uf$ denote
the extension of $f$ by 0 outside $U.$ 

\begin{lemma} \label{Lem4}
Let $\psi \in L^2(\C).$   Then
$\lim_{r\to 1}\|E_{\mathbb{D}_r}P^{\mathbb{D}_r}\psi 
-E_{\mathbb{D}}P^{\mathbb{D}}\psi\|_{L^2(\C)} = 0.$
\end{lemma}

\begin{proof}
Since $\psi$ is  square integrable, for every $\varepsilon>0$ there exists
$\delta>0$ such that $|r-1|<\delta$ implies that
$\|\psi\|_{L^2(\mathbb{D}_{1+\delta}\setminus\mathbb{D}_{1-\delta})}<\varepsilon/2.$ Then
\[\|P^{\mathbb{D}_r}(T_{\mathbb{D}_r\setminus\mathbb{D}_{1-\delta}}\psi)\|_{
L^2(\mathbb{D}_r)}+
\|P^{\mathbb{D}}(T_{\mathbb{D}\setminus\mathbb{D}_{1-\delta}}\psi)\|_{
L^2(\mathbb{D}) } \leq
2\|\psi\|_{L^2(\mathbb{D}_{1+\delta}\setminus\mathbb{D}_{1-\delta})}\leq\varepsilon\]
for $|r-1|<\delta.$ Next the proof of the lemma will be completed by showing that  
\[\|E_{\mathbb{D}_r}P^{\mathbb{D}_r}(T_{\mathbb{D}_{1-\delta}}\psi)- 
E_{\mathbb{D}}P^{\mathbb{D}}(T_{\mathbb{D}_{1-\delta}}\psi)\|_{L^2(\C)}\to
0  \text{ as } r\to 1.\]
We define $G_r(z,w)=F^r(z,w)-F^1(z,w)$ for
$(z,w)\in \C\times \mathbb{D}_{1-\delta},$ where 
\[F^r(z,w)=\frac{T_{\mathbb{D}_r}(z)r^2}{(r^2-z\overline{w})^2}\]
and  $r>1-\delta.$ We note that  $\frac{r^2}{\pi(r^2-z\overline{w})^2}
$ is the Bergman kernel for $\mathbb{D}_r.$ Then there exists $r_0>1$ such that
$G_r\to 0$ uniformly on 
 $\overline{\mathbb{D}_{r_0}\times \mathbb{D}_{1-\delta}}$ as $r\to 1.$ For
$1-\delta<r<r_0$ we have 
\begin{align*}
\|E_{\mathbb{D}_r}P^{\mathbb{D}_r}(T_{\mathbb{D}_{1-\delta}}\psi)
-&E_{\mathbb{ D}}P^{
\mathbb{D}} (T_{\mathbb{D}_{1-\delta}}\psi)\|_{L^2(\C)}^2 \\
=& \int_{\C}\left|
\int_{\mathbb{D}_{1-\delta}} F^r(z,w)\psi(w)dV(w)
-\int_{\mathbb{D}_{1-\delta}} F^1(z,w)\psi(w)dV(w) \right|^2dV(z)\\
\leq & \int_{\C}
\left(\int_{\mathbb{D}_{1-\delta}} |G_r(z,w)||\psi(w)|dV(w)\right)^2dV(z)\\
\leq &\|\psi\|^2_{L^2(\mathbb{D})}\int_{\mathbb{D}_{r_0}}
\int_{\mathbb{D}_{1-\delta}} |G_r(z,w)|^2 dV(w)dV(z).\\
\end{align*}
Since $G_r\to 0$ uniformly as $r\to 1$ we have  
$\|E_{\mathbb{D}_r}P^{\mathbb{D}_r}(T_{\mathbb{D}_{1-\delta}}\psi)
-E_{\mathbb{D}}P^{\mathbb{D}}(T_{\mathbb{D}_{1-\delta}}\psi)\|_{L^2(\C)}\to
0$ as $r\to 1.$ 
\end{proof}

The lemma above and \cite[Lemma 1.4.1]{KrantzBook} imply the following corollary.
\begin{corollary}\label{Cor1}
 Let $\psi \in L^2(\C)$ and $K$ be a compact subset of $\mathbb{D}.$ 
Then $\{P^{\mathbb{D}_r}\psi\}$ converges uniformly to
$P^{\mathbb{D}}\psi$ on $K$ as $r\to 1.$
\end{corollary}

The following lemma  is stated for bounded convex domains because these domains
are the focus of our paper. However, similar ideas can be used for $A^p$ spaces on
starlike domains. This has been done for $A^p(\mathbb{D})$ in 
\cite[Theorem 3, p.30]{DurenSchusterBook}. 

\begin{lemma}\label{Lem5}
 Let $U$ be a bounded convex domain in $\C$ and $f\in A^2(U).$ Then for any
$\varepsilon>0$ there exists a holomorphic polynomial $h$ such that
$\|f-h\|_{L^2(U)}<\varepsilon.$ 
\end{lemma}
\begin{proof} 
Without loss of generality we may assume that  $U$ contains the
origin.  Let us define $f_{r}(z)=f(r z)$ for $r\in (0,1)$ and assume
that $\varepsilon>0$ is given. Then $f_{r}\in A^2(U)\cap C(\overline{U})$ and
one can show that  there exists $0<r<1$ such that
\[\|f-f_{r_0}\|<\frac{\varepsilon}{2}.\] 
This can be seen as follows: First there exists  $0<\delta<1$ so that
$ \|f\|_{L^2(U\setminus \delta U)}<\frac{\varepsilon}{6}.$ The uniform
continuity of $f$ on compact subsets of $U$ implies that there exists $\frac{1}{2}<r<1$
such that 
\[\sup\left\{|f(z)-f(rz)|:z\in\left(\frac{1+\delta}{2}\right)\overline{U}\right\}
<\frac{\varepsilon}{6\sqrt{V(U)}}\]
where $V(U)$ denotes the volume of $U.$  Then we have  
\begin{align*}
\|f-f_{r}\|\leq &
\|f-f_{r}\|_{L^{2}((\frac{1+\delta}{2}) U)}+
\|f\|_{L^2(U\setminus
(\frac{1+\delta}{2})U)} +\|f_{r}\|_{L^{2}(U\setminus
(\frac{1+\delta}{2})U)} \\
\leq &\frac{\varepsilon}{6}+\frac{\varepsilon}{6}+ \frac{1}{r}
\|f\|_{L^2(U\setminus \delta U)} \\
<&\frac{2\varepsilon}{3}
\end{align*}
On the other hand,  Mergelyan's theorem implies that  there exists a
holomorphic polynomial $h$ such that 
\[\sup\{|f_{r}(z)-h(z)|:z\in \overline{U}\} <
\frac{\varepsilon}{3\sqrt{V(U)}}.\] 
 Then we have 
\[\|f-h\|\leq \|f-f_{r}\|+\|f_{r}-h\|\leq \frac{2\varepsilon}{3}+
\frac{\varepsilon}{3}=\varepsilon.\]
This completes the proof of Lemma \ref{Lem5}.
\end{proof}

The next lemma shows that when concentric disks converge, then not
only the kernels and the Bergman projections converge but also the products of
Hankel operators converge ``weakly". 

\begin{lemma}\label{Lem6} 
For $r>0$ let  $\mathbb{D}_{r}=\{z\in \C:|z|<r\}, f_1$ and $f_2$ be entire
functions, and $\phi,\psi\in C(\C).$  Then 
\[\ds \lim_{r\to r_0} \left\langle H^{\mathbb{D}_{r}}_{\phi}(f_{1}),
H^{\mathbb{D}_{r}}_{\psi} (f_{2})\right\rangle_{\mathbb{D}_{r}} =
\left\langle
H^{\mathbb{D}_{r_0}}_{\phi}(f_{1}),H^{\mathbb{D}_{r_0}}_{\psi}
(f_{2})\right\rangle_{\mathbb{D}_{r_0}}.\] 
\end{lemma}

\begin{proof}
First assume that $r_0\leq r.$ For any $0<\delta<r_0$ we have
\begin{align*}
 \left|\left\langle
H^{\mathbb{D}_{r}}_{\phi}(f_{1})\right.\right.&\left.\left.\!\!\!, 
H^{\mathbb{D}_{r}}_{\psi}(f_{2})\right\rangle_{\mathbb{D}_{r}} - 
\left\langle H^{\mathbb{D}_{r_0}}_{\phi}(f_{1}),
H^{\mathbb{D}_{r_0}}_{\psi}(f_{2 }) \right\rangle_{\mathbb{D}_{r_0}}
\right| \\
=&
\left|\left\langle \phi f_{1},H^{\mathbb{D}_{r}}_{\psi}(f_{2})
\right\rangle_{\mathbb{D}_{r}} -
\left\langle \phi f_{1},H^{\mathbb{D}_{r_0}}_{\psi}(f_{2
}) \right\rangle_{\mathbb{D}_{r_0}} \right|\\
\leq & \left|\left\langle \phi f_{1},\psi f_{2}
\right\rangle_{\mathbb{D}_{r}\setminus \mathbb{D}_{r_0}} \right| 
+\left|\left\langle \phi f_{1},P^{\mathbb{D}_{r}}(\psi f_{2})
\right\rangle_{\mathbb{D}_{r}} -
\left\langle \phi f_{1},P^{\mathbb{D}_{r_0}}(\psi f_{2
}) \right\rangle_{\mathbb{D}_{r_0}} \right|\\
\leq&  \left|\left\langle \phi f_{1},\psi f_{2}
\right\rangle_{\mathbb{D}_{r}\setminus \mathbb{D}_{r_0}} \right| +
 \left|\left\langle \phi f_{1},P^{\mathbb{D}_{r}}(\psi f_{2})
\right\rangle_{\mathbb{D}_{r_0-\delta}} - \left\langle \phi
f_{1},P^{\mathbb{D}_{r_0}}(\psi f_{2})
\right\rangle_{\mathbb{D}_{r_0-\delta}} \right|\\
&+ \left|\left\langle \phi f_{1},P^{\mathbb{D}_{r}}(\psi f_{2})
\right\rangle_{\mathbb{D}_{r }\setminus \mathbb{D}_{r_0-\delta}} \right| 
+\left|\left\langle \phi f_{1},P^{\mathbb{D}_{r_0}}(\psi f_{2})
\right\rangle_{\mathbb{D}_{r_0}\setminus \mathbb{D}_{r_0-\delta}} \right| 
\end{align*}
Therefore, we have
\begin{align}\nonumber
 \left|\left\langle
H^{\mathbb{D}_{r}}_{\phi}(f_{1})\right.\right.&\left.\left.\!\!\!, 
H^{\mathbb{D}_{r}}_{\psi}(f_{2})\right\rangle_{\mathbb{D}_{r}} - 
\left\langle H^{\mathbb{D}_{r_0}}_{\phi}(f_{1}),
H^{\mathbb{D}_{r_0}}_{\psi}(f_{2 }) \right\rangle_{\mathbb{D}_{r_0}}
\right|\\ \label{Eqn2}
\leq& \left|\left\langle \phi f_{1},\psi f_{2}
\right\rangle_{\mathbb{D}_{r}\setminus \mathbb{D}_{r_0}} \right|
+ \left|\left\langle \phi f_{1},P^{\mathbb{D}_{r}}(\psi f_{2})
 -  P^{\mathbb{D}_{r_0}}(\psi f_{2})
\right\rangle_{\mathbb{D}_{r_0-\delta}} \right|\\
\nonumber &+\|\phi f_1\|_{L^2(\mathbb{D}_{r}\setminus
\mathbb{D}_{r_0-\delta})}\|\psi f_2\|_{L^2(\mathbb{D}_{r})} 
+\|\phi f_1\|_{L^2(\mathbb{D}_{r_0}\setminus
\mathbb{D}_{r_0-\delta})}\|\psi f_2\|_{L^2(\mathbb{D}_{r_0})}.
\end{align}
Then for $\varepsilon>0$ one can choose $0<\delta_1<\min\{1,r_0\}$ so that 
$\|\phi f_1\|_{L^2(\mathbb{D}_{r_0+\delta_1}\setminus
\mathbb{D}_{r_0-\delta_1})}\leq \varepsilon. $ Furthermore, by Corollary
\ref{Cor1} we can choose $0<\delta_2<\delta_1$ so that 
$r_0\leq r\leq r_0+\delta_2$ implies that 
\[\left|P^{\mathbb{D}_{r}}(\psi f_{2})(z) -  P^{\mathbb{D}_{r_0}}(\psi
f_{2})(z)  \right|\leq \varepsilon \text{ for } z\in
\overline{\mathbb{D}_{r_0-\delta_1}}
\text{ and } \left|\left\langle \phi f_{1},\psi f_{2} 
\right\rangle_{\mathbb{D}_{r}\setminus \mathbb{D}_{r_0}} \right|\leq
\varepsilon.\]
Therefore, for $r_0\leq r\leq r_0+\delta_2$ we have 
\begin{align*}
 \left|\left\langle
H^{\mathbb{D}_{r}}_{\phi}(f_{1}),H^{\mathbb{D}_{r}}_{\psi}(f_{2})
\right\rangle_{\mathbb{D}_{r}} - 
\left\langle H^{\mathbb{D}_{r_0}}_{\phi}(f_{1}),
H^{\mathbb{D}_{r_0}}_{\psi}(f_{2 }) \right\rangle_{\mathbb{D}_{r_0}}
\right|
\leq & 
 \varepsilon \left(1+\|\psi f_2\|_{L^2(\mathbb{D}_{r_0})}
+ \|\psi f_2\|_{L^2(\mathbb{D}_{r})}\right) \\
&+ \left|\left\langle \phi f_{1},P^{\mathbb{D}_{r}}(\psi f_{2})
 -  P^{\mathbb{D}_{r_0}}(\psi f_{2})
\right\rangle_{\mathbb{D}_{r_0-\delta_1}} \right|\\
\leq&  \varepsilon \left(1+\|\psi f_2\|_{L^2(\mathbb{D}_{r_0})}
+ \|\psi f_2\|_{L^2(\mathbb{D}_{r_0+1})}\right)\\
&+ \varepsilon r_0\sqrt{\pi}\|\phi f_{1}\|_{L^2(\mathbb{D}_{r_0})}.
\end{align*}
We note that by the Cauchy-Schwarz inequality we used the following
inequality above
\[\left|\left\langle \phi f_{1},P^{\mathbb{D}_{r}}(\psi f_{2})
 -  P^{\mathbb{D}_{r_0}}(\psi f_{2})
\right\rangle_{\mathbb{D}_{r_0-\delta}} \right| \leq  \varepsilon
r_0\sqrt{\pi}\|\phi f_{1}\|_{L^2(\mathbb{D}_{r_0})}.\]
Therefore, there exists a constant $K>0$ independent of $r$ and $\varepsilon$ so
that 
\[\left|\left\langle
H^{\mathbb{D}_{r}}_{\phi}(f_{1}),H^{\mathbb{D}_{r}}_{\psi}(f_{2})
\right\rangle_{\mathbb{D}_{r}} - 
\left\langle H^{\mathbb{D}_{r_0}}_{\phi}(f_{1}),
H^{\mathbb{D}_{r_0}}_{\psi}(f_{2 }) \right\rangle_{\mathbb{D}_{r_0}}
\right| \leq \varepsilon K_1\]
 for $r_0\leq r\leq r_0+\delta_2.$

Similarly if $r\leq r_0$  equation \eqref{Eqn2} is valid for $r$ and $r_0$
interchanged. For $\varepsilon >0$ we choose  $0<\delta_3<\min\{1,r_0/2\}$
such that $\|\phi f_1\|_{L^2(\mathbb{D}_{r_0}\setminus
\mathbb{D}_{r_0-\delta_3})}<\varepsilon.$ By Corollary \ref{Cor1} we
choose $0<\delta_4<\frac{\delta_3}{2}$ so that  so that $r_0-\delta_4<r\leq r_0$
implies that 
\[\left|P^{\mathbb{D}_{r}}(\psi f_{2})(z) -  P^{\mathbb{D}_{r_0}}(\psi
f_{2})(z)  \right|\leq \varepsilon \text{ for } z\in
\overline{\mathbb{D}_{r_0-\frac{\delta_3}{2}}}
\text{ and } \left|\left\langle \phi f_{1},\psi f_{2} 
\right\rangle_{\mathbb{D}_{r_0}\setminus \mathbb{D}_{r}} \right|\leq
\varepsilon.\]
Therefore, there exists a constant $K_2>0$ independent of $r$ and $\varepsilon$ 
such that 
$r_0-\delta_4\leq r\leq r_0$ implies that 
\[\left|\left\langle
H^{\mathbb{D}_{r}}_{\phi}(f_{1}),H^{\mathbb{D}_{r}}_{\psi}(f_{2})
\right\rangle_{\mathbb{D}_{r}} - \left\langle
H^{\mathbb{D}_{r_0}}_{\phi}(f_{1}),
H^{\mathbb{D}_{r_0}}_{\psi}(f_{2 }) \right\rangle_{\mathbb{D}_{r_0}}
\right| \leq \varepsilon K_2.\]
Thus,  we have  
$\ds \lim_{r\to r_0} \left\langle H^{\mathbb{D}_{r}}_{\phi}(f_{1}),
H^{\mathbb{D}_{r}}_{\psi} (f_{2})\right\rangle_{\mathbb{D}_{r}} =
\left\langle
H^{\mathbb{D}_{r_0}}_{\phi}(f_{1}),H^{\mathbb{D}_{r_0}}_{\psi}
(f_{2})\right\rangle_{\mathbb{D}_{r_0}}.$ 
\end{proof}

\section*{Proof of Theorem \ref{ThmMain}}

\begin{proof}[Proof of Theorem \ref{ThmMain}]
Assume that $H^{*}_{\psi}H_{\phi}$ is a compact operator and there exists an
analytic disk $\Delta$ in $\partial\D$ (if not we are done), and two symbols
$\phi$ and $\psi$ that are not holomorphic ``along'' $\Delta.$ Namely, there
exists a holomorphic function $f:\disk\to\Delta$ so that neither $\phi\circ f$
nor $\psi\circ f$ is holomorphic on $\disk.$ Since $\D$ is a convex Reinhardt 
bounded domain Lemma \ref{Lem2} implies that the disk $\Delta$ is either
horizontal or vertical. So without loss of generality we may assume that
$\Delta$ is horizontal and  
\[\D= \bigcup_{w\in H}\left(\Delta_{w}\times\{w\}\right)\]
 where $H\subset \C, \Delta_w=\{z\in \C:(z,w)\in \D\}$
is a disk in $\C$ centered at the origin, and $\Delta=\Delta_{w_0}$ for some
$w_0\in \partial H.$ By using a linear holomorphic map,
$(z,w)\to(z,e^{i\theta_0}(w-w_0))$ for some $\theta_0\in \mathbb{R},$ we
translate the domain $\D$ into $\{(z,w)\in\C^2:\im(w)<0\}.$ Hence without loss
of generality we may assume that $H\subset \{w\in \C:\im(w)<0 \}$ and 
$\D= \bigcup_{w\in H}\left(\Delta_{w}\times\{w\}\right)$ where 
$\Delta_w$'s are  disks centered at the origin and $\Delta=\Delta_0.$ 

Let us extend $\phi(z,0)$ and $\psi(z,0)$ as continuous functions on $\C$ and call the
extensions $\phi_0(z)$ and $\psi_0(z).$ Since $\phi_0$ and $\psi_0$ are
harmonic and not holomorphic on $\Delta_0,$ Theorem 5 in \cite{Zheng89} (see
also \cite[Corollary 6]{AhernCuckovic01}) implies that the product 
$\left(H^{\Delta_0}_{\psi_0}\right)^*H^{\Delta_0}_{\phi_0}$ is a nonzero
operator. Then there exist $f_1,f_2\in A^2(\Delta_0)$ such that 
\[\int_{\Delta_0}H^{\Delta_0}_{\phi_{0}}(f_{1})(z) \overline{H^{\Delta_0}_{\psi_{0}}
(f_{2})(z)}dV(z) \neq 0.\] 
Then by Lemma \ref{Lem5} we can choose $f_1$ and $f_2$ to be holomorphic
polynomials (of one variable).

For convenience, in the following calculations we will abuse the notation as
follows: we will assume that $\phi_0,\psi_0, f_1,$ and $f_2$ are functions of
$z$ only (or functions of $(z,w)$ but independent of $w$).  We remind the
reader that in the computations below, the Bergman projection on the disk
$\Delta_w$ is denoted by $P^{\Delta_w}$ and 
$H^{\Delta_w}_{\eta} (f)=\eta f- P^{\Delta_w}(\eta f)$ for $f\in A^2(\Delta_w)$ and
$\eta\in L^{\infty}(\Delta_w).$ We note that functions  $(z,w)\to
P^{\Delta_w}(\eta f)(z)$ and $(z,w)\to H^{\Delta_w}_{\eta}(f)(z)$ are 
continuous on $\D.$ In case of the first function this can be seen as follows: 
\begin{align*}
|P^{\Delta_{w_0}}(\eta f)(z_0)-P^{\Delta_w}(\eta f)(z)| \leq&
|P^{\Delta_{w_0}}(\eta f)(z_0)-P^{\Delta_{w}}(\eta f)(z_0)|\\
&+\int_{\Delta_w}|K_{\Delta_w}(z_0,\xi)-K_{\Delta_w}(z,\xi)||\eta(\xi)f(\xi)|dV(\xi).
\end{align*}
As $(z,w)$ goes to $(z_0,w_0)$ in $\D,$ the first term on the right hand side goes to zero
by Corollary \ref{Cor1} and the second term goes to zero because
$\sup\{|K_{\Delta_w}(z_0,\xi)-K_{\Delta_w}(z,\xi)|:\xi\in \Delta_w\}$ goes to zero. 
Also Fubini's Theorem implies that these functions are square integrable. 

Let  $g_j\in A^2(H)$ which will be specified later. For fixed $w\in H$ and any
$z\in \Delta_w$
\[H^{\D}_{\phi_0}(f_1g_j)(z,w)=
\phi_0(z,w)f_1(z)g_j(w)-P^{\D}(\phi_0f_1g_j)(z,w)\] 
and 
\[H^{\Delta_w}_{\phi_0(.,w)}(f_1)(z) = 
\phi_0(z,w)f_1(z)-P^{\Delta_w}(\phi_0(., w)f_1)(z)\] 
 imply that
\[ H^{\D}_{\phi_{0}}(f_{1}g_{j})(z,w)-g_{j}(w)H^{\Delta_w}_{\phi_{0}}(f_{1}) =
P^{\D}(\phi_{0}f_{1}g_{j})(z,w)-g_{j}(w)P^{\Delta_w}(\phi_{0}f_{1})(z)\] 
is holomorphic  in $z$ on $\Delta_w.$ 

Using Lemma \ref{Lem1} in the first equality below we get 
\begin{align*}
\int_{\Delta_w}(H^{\D}_{\psi_0})^*H^{\D}_{\phi_{0}}(f_{1}g_{j}))(z,w)
\overline{f_{2}(z)}dV(z)
=&\int_{\Delta_w}P^{\D}(\overline\psi_{0}
H^{\D}_{\phi_{0}}(f_{1}g_{j}))(z,w)\overline{
f_{2}(z)}dV(z)\\
=&\int_{\Delta_w}\overline{\psi_{0}(z,w)}
H^{\D}_{\phi_{0}}(f_{1}g_{j})(z,w)\overline{f_{2}(z)}dV(z)\\
&-\int_{\Delta_w}(I-P^{\D})(\overline\psi_{0}
H^{\D}_{\phi_{0}}(f_{1}g_{j}))(z,w)\overline{f_{2}(z)}dV(z) \\
=&\int_{\Delta_w}H^{\D}_{\phi_{0}}(f_{1}g_{j})(z,w) \overline{
{H^{\Delta_w}_{\psi_{0}}
(f_{2}})(z)}dV(z)\\
&+\int_{\Delta_w}H^{\D}_{\phi_{0}}(f_{1}g_{j})(z,w)
\overline {P^{\Delta_w} (\psi_{0} f_{2})(z)}dV(z)\\
&-\int_{\Delta_w}(I-P^{\D})(\overline\psi_{0}
H^{\D}_{\phi_{0}}(f_{1}g_{j}))(z,w)\overline{f_{2}(z)}dV(z) \\
=& g_{j}(w)\int_{\Delta_w}H^{\Delta_w}_{\phi_{0}}(f_{1})(z)\overline
{H^{\Delta_w}_{\psi_{0}} (f_{2})(z)}dV(z)\\
&+\int_{\Delta_w}H^{\D}_{\phi_{0}}(f_{1}g_{j})(z,w)
\overline {P^{\Delta_w} (\psi_{0} f_{2})(z)}dV(z)\\
&-\int_{\Delta_w}(I-P^{\D})(\overline\psi_{0}
H^{\D}_{\phi_{0}}(f_{1}g_{j}))(z,w)\overline{f_{2}(z)}dV(z).
\end{align*}
If we multiply both sides by $\overline{g_{j}(w)}$ and integrate over $H$ we
get 
\begin{align*}
\left\langle H^{\D}_{\phi_{0}}(f_{1}g_{j}), H^{\D}_{\psi_{0}}(f_{2}g_{j})
\right\rangle  = & \int_{H}|g_{j}(w)|^{2}\int_{\Delta_w}
H^{\Delta_w}_{\phi_{0}}(f_{1})(z)\overline {H^{\Delta_w}_{\psi_{0}}
(f_{2})(z)}dV(z)dV(w)\\
&+ \int_{\D}H^{\D}_{\phi_{0}}(f_{1}g_{j})(z,w)
\overline{{P^{\Delta_w} (\psi_{0} f_{2})(z)}g_j(w)} dV(z,w)\\
&-\int_{\D}(I-P^{\D})(\overline\psi_{0}
H^{\D}_{\phi_{0}}(f_{1}g_{j}))(z,w)\overline{f_{2}(z)g_j(w)}dV(z,w).
\end{align*}
We note that the last integral on the right hand side above is zero. Hence, we have 
\begin{align}\nonumber
\left\langle H^{\D}_{\phi_{0}}(f_{1}g_{j}), H^{\D}_{\psi_{0}}(f_{2}g_{j})
\right\rangle  = & \int_{H}|g_{j}(w)|^{2}\int_{\Delta_w}
H^{\Delta_w}_{\phi_{0}}(f_{1})(z)\overline {H^{\Delta_w}_{\psi_{0}}
(f_{2})(z)}dV(z)dV(w)\\
\label{Eqn3} &+ \int_{\D}H^{\D}_{\phi_{0}}(f_{1}g_{j})(z,w)
\overline{{P^{\Delta_w} (\psi_{0} f_{2})(z)}g_j(w)} dV(z,w).
\end{align}

Our next goal is to show that the second integral on the right hand side of
\eqref{Eqn3} goes to zero while the first one does not as $j$ goes to infinity.

Let $h$ be an entire function on $\C.$ Then 
\begin{align*}
 \int_{\D}H^{\D}_{\phi_{0}}(f_{1}g_{j})(z,w) &\overline{P^{\Delta_w} (\psi_{0}
f_{2})(z)g_j(w)}dV(z,w)\\
=&\int_{\D}H^{\D}_{\phi_{0}}(f_{1}g_{j})(z,w) \overline{h(z)
g_j(w)}dV(z,w)\\
 &+\int_{\D}H^{\D}_{\phi_{0}}(f_{1}g_{j})(z,w) \overline{(P^{\Delta_w}
(\psi_{0} f_{2})(z) -h(z))g_j(w)}dV(z,w)\\
 =&\int_{\D}H^{\D}_{\phi_{0}}(f_{1}g_{j})(z,w) \overline{(P^{\Delta_w}
(\psi_{0} f_{2})(z) -h(z))g_j(w)}dV(z,w).
 \end{align*} 
Using the Cauchy-Schwarz inequality we have
\begin{multline*}
 \int_{\D}|H^{\D}_{\phi_{0}}(f_{1}g_{j})(z,w) \overline{(P^{\Delta_w}
(\psi_{0} f_{2})(z) -h(z))g_j(w)}dV(z,w)|\\
\leq  \|H^{\D}_{\phi_{0}}(f_{1}g_{j})\|
\left(\int_{\D}|(P^{\Delta_w}
(\psi_{0} f_{2})(z) -h(z))g_j(w)|^2dV(z,w)\right)^{1/2}.
\end{multline*}

Now we choose  $g_j(w)=\frac{a_j}{w^{\alpha_j}}$   such that $a_j\to 0,\alpha_j\to 1^-,$
and $\|g_j\|_H=1$ as $j\to \infty.$  Then  one can show that 
\begin{align}\label{Eqn4}
\|H^{\D}_{\phi-\phi_0}(g_j)\|\leq \|(\phi-\phi_0)g_j\|\to 0  \text{ as } j\to \infty
\end{align}
because $g_{j}$ goes to $0$ uniformly on any compact set  away from
$\Delta_0$ and $\phi-\phi_0=0$ on $\Delta_0.$  
 
Let $\varepsilon>0$ be fixed. Then there exists  a set $L_{\varepsilon}\Subset \Delta_0$  
such that $\|\psi_0f_2\|_{L^2(\Delta_0\setminus L_{\varepsilon})}\leq \varepsilon/2.$
Furthermore, Lemma \ref{Lem5} and \cite[Proposition 1.4.1]{KrantzBook} imply that
there exists an entire function $h$ such that  
\[\|P^{\Delta_0}(\psi_{0} f_{2})-h\|_{L^2(\Delta_0)} \leq \varepsilon \text{ and }
\sup\{|P^{\Delta_0}(\psi_{0} f_{2})(z)-h(z)| :z\in
\overline{L_{\varepsilon}}\} \leq \varepsilon/2.\]
 Then by Lemma \ref{Lem3} we can choose
$\delta_1>0$ such that  $|w|<\delta_1$ implies that $L_{\varepsilon}\Subset \Delta_w.$
Furthermore, $\delta_1$ can be chosen so that  
\[\|\psi_0f_2\|_{L^2(\Delta_w\setminus \Delta_0)}+\|h\|_{L^2(\Delta_w\setminus
\Delta_0)}\leq \varepsilon/2.\]
Finally, Lemma \ref{Lem3} and Corollary \ref{Cor1} imply that there exists
$\delta_2>0$ such that 
\[\sup\{|P^{\Delta_0}(\psi_{0} f_{2})(z)-P^{\Delta_w}(\psi_{0} f_{2})(z)| :z\in
\overline{L_{\varepsilon}}\} \leq \varepsilon/2,\] 
for $|w|<\delta_2$ and  Lemma \ref{Lem3} and Lemma \ref{Lem4} imply that there
exists $\delta_3>0$ such that 
\[\|E_{\Delta_w}P^{\Delta_w}(T_{L_{\varepsilon}}\psi_{0} f_{2})-
E_{\Delta_0}P^{\Delta_0}(T_{L_{\varepsilon}}\psi_{0} f_{2}) 
\|_{L^2(\C)}  \leq   \varepsilon\] 
 for $|w|<\delta_3.$ 
 
If we put all these together we have the following: for $\varepsilon>0$ there exist  
$\delta=\min\{\delta_1,\delta_2,\delta_3\}>0,$ a set $L_{\varepsilon}\Subset
\Delta_0,$  and an entire function $h$  such that $|w|<\delta$ implies that 

\begin{itemize}
\item[i.] $L_{\varepsilon}\Subset \Delta_w, \|\psi_{0}
f_{2}\|_{L^2(\Delta_w\setminus L_{\varepsilon})} \leq
\varepsilon,$ and  $\|h\|_{L^2(\Delta_w\setminus \Delta_0)}\leq \varepsilon/2,$
\item[ii.] $\|P^{\Delta_0}(\psi_{0} f_{2})-h\|_{L^2(\Delta_0)} \leq
\varepsilon,$ 
\item[iii] $ \sup\{|P^{\Delta_w}(\psi_{0} f_{2})(z)-h(z)| :z\in
\overline{L_{\varepsilon}}\} \leq \varepsilon,$ 
\item[iv.]  $\|E_{\Delta_w}P^{\Delta_w}(T_{L_{\varepsilon}}\psi_{0}
f_{2})- E_{\Delta_0}P^{\Delta_0}(T_{L_{\varepsilon}}\psi_{0} f_{2})
\|_{L^2(\C)} 
 \leq   \varepsilon.$ 
\end{itemize}

Now we choose $j_0$ so that 
\[|g_j(w)|<\varepsilon
\left(1+\int_{\D}|P^{\Delta_w}(\psi_{0}f_{2})(z)-h(z)|^2dV(z,w)\right)^{-\frac{1}{2}} \]
for $|w|\geq \delta$ and $j\geq j_0.$ Let us define
 $K_{\delta}=\cup_{|w|\geq \delta} \Delta_w \subset \D.$  
Then we have 
\begin{align*}
\int_{\D}|(P^{\Delta_w}&(\psi_{0} f_{2})(z)-h(z)) g_j(w)|^2dV(z,w)\\
=&\int_{K_{\delta}}|(P^{\Delta_w}(\psi_{0} f_{2})(z)-h(z)) g_j(w)|^2dV(z,w)\\
&+\int_{\D\cap (L_{\varepsilon}\times B(0,\delta))}|(P^{\Delta_w}(\psi_{0}
f_{2})(z)-h(z))
g_j(w)|^2dV(z,w)\\
&+ \int_{\D\setminus (K_{\delta} \cup (L_{\varepsilon}\times
B(0,\delta))}|(P^{\Delta_w}(\psi_{0} f_{2})(z)-h(z))
g_j(w)|^2dV(z,w)\\
 \lesssim& \sup\{|g_j(w)|^2:|w|\geq \delta\} \int_{K_{\delta}}
|P^{\Delta_w}(\psi_{0} f_{2})(z)-h(z)|^2dV(z,w)\\
&+\sup\{|P^{\Delta_w}(\psi_{0} f_{2})(z)-h(z)|^2 :z\in
\overline{L_{\varepsilon}},|w|\leq
\delta \}\int_{H}|g_j(w)|^2dV(w)\\
&+ \int_{|w|<\delta}|g_j(w)|^2\int_{\Delta_w\setminus L_{\varepsilon}}
|P^{\Delta_w}(\psi_{0} f_{2})(z)-h(z)|^2dV(z)dV(w)\\
\lesssim & \varepsilon^2 
+\int_{|w|<\delta}|g_j(w)|^2\int_{\Delta_w\setminus L_{\varepsilon}}
|P^{\Delta_w}(\psi_{0} f_{2})(z)-h(z)|^2dV(z)dV(w).
\end{align*}
We note that iii. is used in the last inequality. Then
\begin{align*}
\int_{\Delta_w\setminus L_{\varepsilon}} &|P^{\Delta_w}(\psi_{0}
f_{2})(z)-h(z)|^2dV(z)\\
\lesssim &\int_{\Delta_w\setminus L_{\varepsilon}} 
|P^{\Delta_w}((1-T_{L_{\varepsilon}})\psi_{0}
f_{2})(z)|^2dV(z)\\
&+\int_{\Delta_w\setminus L_{\varepsilon}}
|P^{\Delta_w}(T_{L_{\varepsilon}}\psi_{0}
f_{2})(z)-E_{\Delta_0}P^{\Delta_0}(T_{L_{\varepsilon}}\psi_{0}
f_{2})(z)|^2dV(z)\\
&+\int_{\Delta_w\setminus L_{\varepsilon}}
|E_{\Delta_0}P^{\Delta_0}(T_{L_{\varepsilon}}\psi_{0}
f_{2})(z)-E_{\Delta_0}P^{\Delta_0}(\psi_{0} f_{2})(z)|^2dV(z)\\
&+ \int_{\Delta_w\setminus L_{\varepsilon}} |E_{\Delta_0}P^{\Delta_0}(\psi_{0}
f_{2})(z)-h(z)|^2dV(z).
\end{align*}
Let  $|w|<\delta.$ Then by i. we have 
\[\|P^{\Delta_w}((1-T_{L_{\varepsilon}})\psi_{0}
f_{2})\|^2_{L^2(\Delta_w\setminus
L_{\varepsilon})} \leq \|(1-T_{L_{\varepsilon}})\psi_{0}
f_{2})\|^2_{L^2(\Delta_w)} 
\leq \varepsilon^2\]
and by iv. we have 
\begin{align*}
\int_{\Delta_w\setminus L_{\varepsilon}}
&|P^{\Delta_w}(T_{L_{\varepsilon}}\psi_{0}
f_{2})(z)-E_{\Delta_0}P^{\Delta_0}(T_{L_{\varepsilon}}\psi_{0} f_{2})(z)|^2dV(z)
\\
\leq  &
\|E_{\Delta_w}P^{\Delta_w}(T_{L_{\varepsilon}}\psi_{0} f_{2})-
E_{\Delta_0}P^{\Delta_0}(T_{L_{\varepsilon}}\psi_{0} f_{2}) \|^2_{L^2(\C)}\\
 \leq  & \varepsilon^2.
\end{align*}
By i. again and the fact that $\Delta_0\subset \Delta_w$ we have 
\[\int_{\Delta_w\setminus L_{\varepsilon}}
|E_{\Delta_0}P^{\Delta_0}(T_{L_{\varepsilon}}\psi_{0}
f_{2})(z)-E_{\Delta_0}P^{\Delta_0}(\psi_{0} f_{2})(z)|^2dV(z) 
\leq \|(1-T_{L_{\varepsilon}})\psi_{0} f_{2})\|^2_{L^2(\Delta_0)}\leq
\varepsilon^2.\] 
Furthermore, i. and  ii. imply that 
\begin{align*}
 \int_{\Delta_w\setminus L_{\varepsilon}} |E_{\Delta_0}P^{\Delta_0}(\psi_{0}
f_{2})(z)-h(z)|^2dV(z)=&\int_{\Delta_0\setminus L_{\varepsilon}}
|P^{\Delta_0}(\psi_{0}
f_{2})(z)-h(z)|^2dV(z)\\
&+\int_{\Delta_w\setminus\Delta_0} |h(z)|^2dV(z)\\
\lesssim& \varepsilon^2
\end{align*}
Therefore, we have 
\[\int_{\Delta_w\setminus L_{\varepsilon}} |P^{\Delta_w}(\psi_{0}
f_{2})(z)-h(z)|^2dV(z)
\lesssim \varepsilon^2 \text{ for } |w|\leq \delta .\]
Furthermore, since $\int_H|g_j(w)|^2dV(w)=1$ and $\D$ is bounded there exists a
constant $C>0$  such that $\|g_j\|<C$. Therefore, 
\[\int_{\D}H^{\D}_{\phi_{0}}(f_{1}g_{j})(z,w) \overline{P^{\Delta_w} (\psi_{0}
f_{2})(z)g_j(w)}dV(z,w)\to 0 \text{ as } j\to\infty.\]
  
Now we will show that the first integral on the right hand side of
\eqref{Eqn3} stays away from zero as $j$ goes to infinity. We remind the
reader that $f_1$ and $f_2$ are holomorphic polynomials such that
$\int_{\Delta_0}H^{\Delta_0}_{\phi_{0}}(f_{1})(z) 
\overline {H^{\Delta_0}_{\psi_{0}} (f_{2})}(z)dV(z)\neq 0.$ 
Therefore, by Lemma \ref{Lem6}, without loss of generality and by choosing a
smaller $\delta>0,$ if necessary, we may assume that  there exists $\beta>0$ such that 
\[\text{Re}\left(\int_{\Delta_w}H^{\Delta_w}_{\phi_{0}}(f_{1})(z) \overline
{H^{\Delta_w}_{\psi_{0}} (f_{2})(z)}dV(z)\right) > \beta \]
for $|w|<\delta.$ The mass of $g_j$ ``accumulates" at the origin in the sense that 
$\int_{H}|g_{j}(w)|^{2}dV(w)=1$ for all $j$ while  $g_j(w)\to 0$ as $w$ stays away 
from $\Delta_0.$ Then  there exists $j_0$ so that  $j\geq j_0$ implies that 
\[\left|\int_{K_{\delta}}|g_{j}(w)|^{2}H^{\Delta_w}_{\phi_{0}}
(f_{1})(z)\overline{H^{\Delta_w}_{\psi_{0}}  (f_{2})(z)}dV(z)dV(w)\right| <\beta/4.\]
On the other hand, there exists $j_1$ such that 
\begin{align*}
\text{Re}\left(\int_{\D\setminus K_{\delta}}|g_{j}(w)|^{2}
H^{\Delta_w}_{\phi_{0}}(f_{1})(z) \overline{H^{\Delta_w}_{\psi_{0}}
(f_{2})(z)}dV(z,w)\right) >&\beta \int_{\{w\in H:|w|<\delta\}}|g_j(w)|^2dV(w)\\
 >&\beta/2
\end{align*}
for $j\geq j_1.$ Therefore, for $j\geq \max\{j_0,j_1\}$ we have 
\[\text{Re}\left(\int_{\D}|g_{j}(w)|^{2} H^{\Delta_w}_{\phi_{0}}(f_{1})(z)
\overline{H^{\Delta_w}_{\psi_{0}} (f_{2})(z)}dV(z,w)\right)>\beta/2.\]
This shows that  the first integral on the right hand side of \eqref{Eqn3} stays away
from zero. Hence, by \eqref{Eqn3} again, $\left\langle
H^{\D}_{\phi_0}(f_{1}g_{j}),H^{\D}_{\psi_0}(f_{2}g_{j})\right\rangle$ does not converge
to zero as $j$ goes to infinity.

Now we we will show that 
$\left\langle \left(H^{\D}_{\psi}\right)^*H^{\D}_{\phi}(f_{1}g_{j}),
f_{2}g_{j}\right\rangle$ does not converge to zero which contradicts the assumption that 
$H_{\psi}^*H_{\phi}$ is compact.  
\begin{align*}
\left|\left\langle \left(H^{\D}_{\psi}\right)^* H^{\D}_{\phi}(f_{1}g_{j}),
f_{2}g_{j}\right\rangle \right|=& \left|\left\langle
H^{\D}_{\phi}(f_{1}g_{j}),H^{\D}_{\psi}(f_{2}g_{j})\right\rangle \right|\\
\lesssim &\left|\left\langle H^{\D}_{\phi_{0}}(f_{1}g_{j}),H^{\D}_{\psi_{0}}(f_{2}g_{j}
\right\rangle\right| +\|(\phi-\phi_0)f_{1}g_{j}\|
\|\psi_0f_2g_j\|\\
&+\|\phi_0f_1g_j\|\|(\psi-\psi_0)f_{2}g_{j}\|+ \|(\phi-\phi_0)f_{1}g_{j}\|
\|(\psi-\psi_0)f_{2}g_{j}\|
\end{align*}
We note that by \eqref{Eqn4} the last three terms on the right hand side of the inequality
above go to zero as $j$ goes to $\infty$ and we just showed that the first term stays
away from zero.  Hence,  
$\left\langle \left(H^{\D}_{\psi}\right)^*H^{\D}_{\phi}(f_{1}g_{j}),f_{2}g_{j}
\right\rangle$ does not converge to zero.
\end{proof}

\singlespace

\end{document}